\documentclass{article}

\usepackage[leqno]{amsmath}
\usepackage{amsfonts}
\usepackage{amssymb}
\usepackage{amsthm}

\def\C{\subset}
\def\CL#1{\overline {#1}}
\def\e{\emptyset}
\def\mi{\backslash}

\def\o{\omega}
\def\cont{\mathfrak{c}}
\def\pritem#1{\noindent\hangafter=1
\hangindent=.27in\rlap{#1}\hskip.27in}
\def\LSS{Lindel\"of $\Sigma$}

\def\calP{{\mathcal P}}
\def\calU{{\mathcal U}}
\def\calV{{\mathcal V}}
\def\calF{{\mathcal F}}
\def\calC{{\mathcal C}}
\def\calN{{\mathcal N}}

\def\amsN{ \mathbb N}
\def\amsD{ \mathbb D}

\begin{document}

\swapnumbers
\newtheorem{defi}{Definition}[section]
\newtheorem{thm}[defi]{Theorem}
\newtheorem{exa}[defi]{Example}
\newtheorem{pro}[defi]{Proposition}
\newtheorem{lem}[defi]{Lemma}
\newtheorem{que}[defi]{Question}
\newtheorem{cor}[defi]{Corollary}

\title{Weakly   linearly Lindel\"of  \\
monotonically   normal   spaces  \\ are
Lindel\"of}

\author{I. Juh\'asz${}^1$, V.V. Tkachuk${}^{2}$,
R.G. Wilson${}^{2}$ }

\maketitle
\footnotetext[1]{Research  supported  by  NKFIH (OTKA) grant
no. 113047 (Hungary)}
\footnotetext[2]{Research  supported  by  CONACyT  grant
CB-2012-01-178103 (Mexico)}

\begin{abstract}
We call a space $X$ {\em weakly  linearly  Lindel\"of  }  if  for  any
family  $\calU$  of  non-empty  open  subsets  of  $X$ of regular
uncountable cardinality $\kappa$, there  exists  a point $x\in X$
such that every neighborhood of $x$ meets $\kappa$-many  elements
of  $\calU$.   We also introduce the concept of {\em almost discretely
Lindel\"of} spaces
as the ones in which every discrete subspace can be covered by
a Lindel\"of subspace. We prove that, in addition to
linearly Lindel\"of spaces, both
weakly  Lindel\"of  spaces and  almost discretely Lindel\"of spaces
are weakly linearly Lindel\"of.

The main result of the paper
is formulated in the title.  It implies, among other things, that
every weakly Lindel\"of monotonically normal space is Lindel\"of;
this result seems to be new even for linearly ordered topological  spaces.

We   show   that,   under   the   hypothesis  $2^\o<  \o_\o$,  if
the co-diagonal
$\Delta^c_X=(X\times X)\mi  \Delta_X$ of a space $X$ is  discretely Lindel\"of,
then  $X$  is  Lindel\"of  and  has  a  weaker  second  countable
topology; here $\Delta_X=\{(x,x):  x\in X\}$ is the  diagonal  of
the  space $X$.  Moreover, the discrete Lindel\"ofness of $\Delta^c_X$
together with the Lindel\"of  $\Sigma$-property of $X$ imply that
$X$ has a countable network. \end{abstract}

\smallskip\noindent
{\small\bf 2010 Mathematics Subject Classification}

\smallskip\noindent
{\small\bf Primary:}  54D20,

\smallskip\noindent
{\small \bf Secondary:} 54A25, 54F05


\section{Introduction}

The closures of discrete  sets  determine quite a few topological
properties of a space $X$.  For example, if $\CL  D$  is  compact
for any discrete $D\C X$, then $X$ is compact \cite{Tk1}.  If $\CL
D$ is linearly (hereditarily) Lindel\"of for each discrete subset
$D\C X$, then $X$  is  linearly (hereditarily) Lindel\"of as well
\cite{ATW}.  In case when $X$ is compact, countable character  of
the  closures  of  all  discrete  subsets  of  $X$  implies  that
$\chi(X)\leq\o$; this was  proved in \cite{ATW}.

If $\calP$ is a topological property, it is said that a space $X$
is {\it  discretely  $\calP$}  if  $\CL  D$  has  $\calP$ for any
discrete set $D\C X$. Thus, every  discretely  compact  space  is
compact.   However,  it  is  an  open  problem  of  Arhangel'skii
\cite{Ar2}   whether   every   discretely   Lindel\"of  space  is
Lindel\"of. It is easy  to  see  that a linearly Lindel\"of space
$X$ must be Lindel\"of if $l(X)<\o_\o$.  Besides, it is a  result
of  Arhangel'skii  and  Buzyakova  \cite{AB}  that any discretely
Lindel\"of space of countable tightness is Lindel\"of.

Since there  is  still  a  possibility  that  not  all discretely
Lindel\"of spaces are Lindel\"of, a natural line of  research  is
to  find  out  in  which  classes discrete Lindel\"ofness implies
Lindel\"ofness and  to  try  to  prove  for discretely Lindel\"of
spaces the classical results known for Lindel\"of ones.  In  this
spirit,  it  was  proved  in  \cite{TW}  that  every   discretely
Lindel\"of monotonically normal space is Lindel\"of.  If $X$ is a
Tychonoff space and $\Delta_X=\{(x,x): x\in X\}$ is its diagonal,
then the Lindel\"of property of the set $\Delta^c_X= (X\times  X)
\mi  \Delta_X$  implies that $X$ is Lindel\"of and $iw(X)\leq\o$,
i.e.,  $X$  has   a   weaker   second   countable  topology  (see
\cite[Theorem 2.1.8]{Ar1}).  Clearly, it would be interesting  to
prove   the  same  for  spaces  $X$  such  that  $\Delta^c_X$  is
discretely  Lindel\"of.   The   respective  open  questions  were
formulated in  \cite{AJTW}  and  \cite{BT}.   It  is  also  worth
mentioning  that  it is an open question (attributed in \cite{Pe}
to Arhangel'skii and Buzyakova)  whether linear Lindel\"ofness of
$\Delta^c_X$ for a compact $X$ implies that $X$ is metrizable.

Burke and Tkachuk established in \cite{BT} that for any countably
compact  space  $X$,   discrete  Lindel\"ofness  of  $\Delta^c_X$
implies that $X$ is compact  and  metrizable.  It  was  asked  in
\cite{AJTW} whether  the  same  is  true  if  the  space  $X$  is
pseudocompact         and $\Delta^c_X$    is    discretely
$\sigma$-compact.   In   this   paper   we   show  that  discrete
$\sigma$-compactness   of   $\Delta^c_X$   implies    $hl(X\times
X)\leq\o$;  it  is easy to deduce from this fact that the answers
to Questions 5.8 and 5.9 of  the paper \cite{AJTW} are positive.
We show that, under $2^\cont<\o_\o$, any Tychonoff space $X$ such
that $\Delta^c_X$ is discretely Lindel\"of must be Lindel\"of and
has countable  $i$-weight.   Besides  it  is  true  in  ZFC  that
discrete  Lindel\"ofness  of  $\Delta^c_X$ implies that $X$ has a
small  diagonal.   In   particular,   if   $X$  is  a  Lindel\"of
$\Sigma$-space and $\Delta^c_X$ is  discretely  Lindel\"of,  then
the  space $X$ is cosmic; this answers Problem 4.6 from the paper
\cite{BT}.

We  also  introduce  the  classes of almost discretely Lindel\"of
spaces and weakly linearly  Lindel\"of  spaces. It turns out that
these classes have  nice  properties;  besides,  any
weakly  linearly  Lindel\"of  and monotonically normal space is
Lindel\"of.  This result, which we  consider to be interesting in
itself, seems to be  new  even  for  weakly  Lindel\"of  linearly
ordered topological spaces.

\section{Notation and terminology}

All  spaces  are  assumed  to  be  $T_1$.  Given a space $X$, the
family       $\tau(X)$       is       its       topology      and
$\tau^*(X)=\tau(X)\mi\{\e\}$; besides,  $\tau(x,X)=\{U\in\tau(X):
x\in    U\}$    for    any    $x\in   X$;   if   $A\C   X$   then
$\tau(A,X)=\{U\in\tau(X):  A\C U\}$.  All ordinals are identified
with the set of their  predecessors  and are assumed to carry the
order topology.  We denote by  $\cont$ the cardinal $2^\o$, by
$\amsD$  the  set  $\{0,1\}$  with  the  discrete  topology   and
$\amsN=\o\mi\{0\}$.  If  $X$  is a space then $\Delta_X=\{(x,x):
x\in X\}\C X\times X$ is the diagonal of $X$. A space $X$ is said
to have a {\it small  diagonal}  if  for any uncountable set $A\C
\Delta^c_X=  (X\times  X)  \mi   \Delta_X$,   there   exists   an
uncountable $B\C A$ such that $\CL B\cap \Delta_X=\e$.

The  cardinal $l(X)=\min\{\kappa: $ every open cover of $X$ has a
subcover  of  cardinality  at most $ \kappa\}$ is called the {\it
Lindel\"of number of $X$} and $hl(X)=\sup\{l(Y):Y\C X\}$  is  the
hereditary  Lindel\"of number of $X$. A space $X$ is called { \it
Lindel\"of if $l(X)\leq\o$}. If for every open cover $\calU$ of a
space $X$ there  exists  a  countable  $\calU'\C \calU$ such that
$\bigcup\calU'$ is dense in $X$, then the  space  $X$  is  called
{\it  weakly  Lindel\"of.} A space $X$ is called {\it generalized
ordered space or simply  GO  space}  if  $X$ is homeomorphic to a
subspace of a linearly ordered space.

We  say that a family $\calF$ of subsets of a space $X$ is {\it a
network modulo a cover} $\calC$  if for any $C\in\calC$ and $U\in
\tau(C,X)$ there exists $F\in \calF$ such that $C\C F \C  U$.   A
Tychonoff space $X$ is {\it Lindel\"of $\Sigma$} (or has {\it the
Lindel\"of $\Sigma$-property}) if there exists a countable family
$\calF$ of subsets of $X$ such that $\calF$ is a network modulo a
compact  cover  $\calC$  of the space $X$.  A space $X$ is called
{\it monotonically normal} if  it  admits an operator $O$ (called
the {\it monotone normality operator}) that assigns to any  point
$x\in  X$  and  any  $U\in \tau(x,X)$ a set $O(x,U)\in \tau(x,X)$
such that $O(x,U)\C U$  and  for  any  points $x,y\in X$ and sets
$U,V\in \tau(X)$ such that $x\in U$ and $y\in V$, it follows from
$O(x,U) \cap O(y,V)\neq\e$ that $x\in V$ or $y\in U$.

As  usual, we denote by $d(X)$ the minimal cardinality of a dense
subset of  $X$  and  $hd(X)=\sup\{d(Y):   Y\C  X\}$.  The minimal
cardinality of a local base at a point $x\in  X$  is  called  the
{\it  character  of $X$ at $x$}; it is denoted by $\chi(x,X)$ and
$\chi(X)=\sup\{\chi(x,X):x\in X\}$. If $X$  is  a space and $x\in
X$  then  let  $\psi(x,X)=\min\{|\calU|:  \calU\C  \tau(X)$   and
$\bigcap\calU=\{x\}\}$  and $\psi(X)=\sup\{\psi(x,X):  x\in X\}$;
the cardinal $\psi(X)$ is called the {\it pseudocharacter} of the
space  $X$.   Given  an  infinite  cardinal  $\kappa$ we say that
$t(X)\leq\kappa$ if, for  any  $A\C  X$  and  $x\in  \CL A$ there
exists a set $B\C A$ such that $|B|\leq\kappa$ and $x\in \CL  B$.
For a Tychonoff space $X$, the cardinal $iw(X)=\min\{\kappa:$ the
space $X$ has a weaker Tychonoff topology of weight $\kappa\}$ is
called   the   {\it  $i$-weight}  of  $X$.  The  cardinal  $c(X)=
\sup\{|\calU|: \calU\C  \tau^*(X)$  is  disjoint$\}$  is the {\it
Souslin number  of  $X$};  the  spaces  whose  Souslin  number  is
countable are said to have the {\it Souslin property.}

Given a space $X$, a family  $\calN$  of  subsets  of  $X$  is  a
network  of $X$ if for every $U\in \tau(X)$ there exists a family
$\calN'\C  \calN$  such   that  $U=\bigcup\calN'$.   Furthermore,
$nw(X)=\min\{|\calN|:\calN$ is a network in $X\}$.  The  cardinal
$nw(X)$  is  the  {\it  network weight} of $X$; the spaces with a
countable network are  called  {\it  cosmic}.  If  $\kappa$ is an
infinite  cardinal,  then   a   space   $X$   is   said   to   be
$\kappa$-monolithic if $nw(\CL A) \leq\kappa$ for any set $A\C X$
such that $|A|\leq\kappa$. For a set $A\C X$, we say  that  $x\in
X$  is  a  {\it  complete  accumulation  point of $A$} if $|U\cap
A|=|A|$ for every $U\in \tau(x,X)$.

The rest of our  terminology  is  standard and follows \cite{En};
the survey \cite{Ho} and the  book \cite{Ju} can be consulted for
definitions and properties of cardinal invariants.

\section{The results}

We start by giving some results that involve assumptions of  discrete
Lindel\"ofness type on the co-diagonal.
Our main aim is to show  that  it  is  consistent  with  ZFC that discrete
Lindel\"ofness of the co-diagonal of  a  Tychonoff
space $X$ implies the Lindel\"ofness of $X$ and the existence of a weaker
second countable Tychonoff topology on $X$.

\begin{pro}\label{pro:sigcom}
If $X$  is  a  Tychonoff  space  and  $\Delta^c_X$  is discretely
$\sigma$-compact, then $hl(X\times X)=\o$. \end{pro}

\begin{proof}
If $K$ is a compact subset of $X$, then the space $\Delta^c_K$ is
discretely  $\sigma$-compact  being   a  closed  subset  of
$\Delta^c_X$. Therefore $K$ is metrizable by Proposition  3.3  of
\cite{BT}.  This  shows that all compact subsets in $X$ and hence
in $X\times X$ are metrizable.   If $D\C \Delta^c_X$ is discrete,
then $\CL D$ is the union of countably  many  metrizable  compact
subsets  (the  bar denotes the closure in $\Delta^c_X$) and hence
$hl(\CL D) \leq nw(\CL D) \leq \o$.  It follows from  Proposition
2.1  of \cite{ATW} that $hl(\Delta^c_X)\leq\o$.  Finally, observe
that $hl(X\times X)= hl(\Delta_X^c)\leq\o$. \end{proof}

The following corollary  solves  Questions  5.8  and 5.9 from the
paper \cite{AJTW}.

\begin{cor}
Suppose that $X$ is  a  Tychonoff  space  and
$\Delta^c_X$ is discretely $\sigma$-compact.

\smallskip
\noindent(a)  If  $X$  is  pseudocompact,  then it is compact and
metrizable.

\smallskip
\noindent(b) If  $X$  is  a  \LSS-space,  then  it  is cosmic.
\end{cor}

\begin{proof}
Apply Proposition \ref{pro:sigcom}  to  see  that  in both cases,
$hl(X\times X)\leq\o$ and hence $iw(X)\leq\o$ (see  \cite[Theorem
2.1.8]{Ar1}).   For the case (a) this implies that $X$ is compact
and metrizable by Problem  140  of  the book \cite{Tkb1}. For the
case   (b)   we   can   conclude   that   $nw(X)\leq\o$  applying
\cite[Theorem 2]{Tk2}. \end{proof}

\begin{thm}\label{thm:cdiagdl}
Suppose that $X$ is a Hausdorff space such that  $\Delta^c_X$  is
discretely  Lindel\"of.  Then $hl(X)\leq\cont$ and hence $|X|\leq
2^\cont$.  \end{thm}

\begin{proof}
Take any discrete subspace  $D\C  X$  and consider the set $F=\CL
D$.  Since $\Delta^c_F$ is a closed subspace of $\Delta^c_X$,  it
must  be  discretely  Lindel\"of.  It is straightforward to check that the
set $(D\times D) \mi  \Delta_F$ is discrete and dense in $\Delta^c_F$,
hence $\Delta^c_F$ is Lindel\"of and so  is  $F$.  Using  the
Lindel\"of  property  of $\Delta_F^c$ it is easy to find a family
$\calU=\{U_n,V_n: n\in \o\}$  of  open  subsets  of $F$ such that
$U_n \cap V_n=\e$ for every  $n\in\o$  and  $\bigcup\{U_n  \times
V_n:   n\in\o\}=  \Delta_F^c$.   It  is immediate that the family
$\calU$ is (even $T_2$-)separating,  so  $|F|\leq\cont$ by \cite[Theorem
3.7(a)]{Ho}. This means that $|\CL D|\leq \cont$ holds for any discrete
$D\C  X$  and,  consequently,  $hl(X)\leq\cont$.
Indeed, this is immediate from the fact that $hl(X)$ is also the
supremum of the sizes of all right separated (i.e. scattered)
subspaces of $X$. Then  $|X|\leq
2^{hl(X)}\leq 2^\cont$ follows. \end{proof}

The  following  corollary gives a consistent answer to Problem 4.5
from the paper \cite{BT}.

\begin{cor}
Assume  that  ${\cont}<\o_\o$ and $X$ is a regular space such
that its co-diagonal  $\Delta^c_X$  is   discretely   Lindel\"of.
Then  actually
$\Delta^c_X$ is Lindel\"of  and  hence  $l(X\times  X)=iw(X)=\o$.
\end{cor}

\begin{proof}
It follows from  Theorem  \ref{thm:cdiagdl}  that  $l(\Delta^c_X)
\leq  hl(\Delta^c_X) \leq \cont<\o_\o$.    Every   discretely
Lindel\"of space $Z$ is  Lindel\"of  if $l(Z)<\o_\o$, so
$\Delta^c_X$ is Lindel\"of.  It is immediate from this that $l(X\times  X)=
l(\Delta^c_X)=\o$ .  But then  $X$  is  a regular Lindel\"of  space  with  a
$G_\delta$-diagonal,   which   implies   $iw(X)=\o$   by
\cite[Theorem 2.1.8]{Ar1}. \end{proof}

Next we introduce and study a couple of weakenings of the  discretely
Lindel\"of property.

\begin{defi} \rm
A space $X$ will be called {\it almost discretely Lindel\"of}  if
for  any discrete set $D\C X$, there exists a Lindel\"of set $L\C
X$ such that $D\C L$. \end{defi}

The following proposition lists a few basic properties of this concept.
Its proof is straightforward and so is left to the reader.

\begin{pro}
(a) if $X$ is almost discretely Lindel\"of, then $ext(X)\leq \o$.

\pritem{(b)}
Any  discretely  Lindel\"of  space  is   almost   discretely
Lindel\"of.

\pritem{(c)}
Any space of countable spread if almost discretely Lindel\"of.

\pritem{(d)} Any continuous image of an almost discretely
Lindel\"of space is almost discretely
Lindel\"of.

\end{pro}

\begin{thm}
If $X$ is  an  almost  discretely Lindel\"of Hausdorff space such
that $\psi(X)\leq\o$ and $t(X)\leq\o$, then  $|X| \leq 2^\cont$. \end{thm}

\begin{proof}
If  $D\C  X$ is discrete, then there exists a Lindel\"of subspace
$L$ of the space $X$ such  that $D\C L$. But then it follows from $\psi(L)
\cdot t(L) \leq \psi(X) \cdot t(X)\leq  \o$  that  $|L|\leq\cont$
(see   \cite[Theorem   1.1.10]{Ar1})    and   hence   $|D|   \leq
|L|\leq\cont$. Thus we have shown that
$s(X)\leq\cont$.  But then we may conclude  $|X|\leq  2^{s(X)\cdot\psi(X)}\leq  2^\cont$
applying say 2.15 (a) of \cite{Ju}.
\end{proof}

Coupling the above proof with the argument we used in the proof of Theorem \ref{thm:cdiagdl} we
get the following result.

\begin{cor}
If $X$ is an almost  discretely  Lindel\"of Hausdorff space of character
$\chi(X) \le \o$,
then $hl(X) \leq \cont$ (and hence $|X| \leq
2^\cont$). \end{cor}

We do not know if the upper bound $2^\cont$ for the cardinality can be improved to
$\cont$ in the above two results.

\smallskip

The second new concept we introduce is the {\em weakly  linearly  Lindel\"of} property
that figures in the title of our paper. As we shall see, it is actually a weakening of
the previously treated almost  discretely  Lindel\"of property.
Our aim is to  prove that every monotonically normal and weakly linearly
Lindel\"of   space is Lindel\"of, as is stated in the title.
We think, however, that this new concept is also interesting in itself.

\begin{defi} \label{def:wll}
\rm
We  say that a space $X$ is {\it weakly linearly Lindel\"of} \ if
for any family $\calU\C \tau^*(X)$ such that $\kappa= |\calU|$ is
an uncountable regular cardinal,  we  can  find  a point $x\in X$
such  that  every  $V\in  \tau(x,X)$   intersects   $\kappa$-many
elements  of  $\calU$. Such a point  $x$ is called a {\it complete
accumulation point} of $\calU$. \end{defi}

The following result implies that in the definition
of  weak  linear  Lindel\"ofness we could have restricted ourselves
to {\em disjoint} families $\calU$ of open sets.

\begin{pro}\label{pro:disj}
Let $X$ be any space and assume that $\calU\C \tau^*(X)$ is such that $\kappa= |\calU|$ is
a regular cardinal, moreover $\calU$ has no complete accumulation point.
Then there is a disjoint family $\calV \C \tau^*(X)$ with $\kappa= |\calV|$
such that $\calV$ has no complete accumulation point either.
\end{pro}

\begin{proof}
We are going to define, by transfinite recursion on $\alpha < \kappa$, sets $U_\alpha \in \calU$
and $V_\alpha \in \tau^*(X)$ with $V_\alpha \subset U_\alpha$ such that $\alpha \ne \beta$ implies
both $U_\alpha \ne U_\beta$ and $V_\alpha \cap V_\beta \ne \emptyset$. Clearly, then
$\calV = \{V_\alpha : \alpha < \kappa\}$ is as required.

So, assume that $\alpha < \kappa$ and for every $\beta < \alpha$ we have defined
$U_\beta \in \calU$ and  $V_\beta \in \tau^*(X)$ with $V_\beta \subset U_\beta$
with the additional property that $$|\{U \in \calU : U \cap V_\beta \ne \emptyset\}| < \kappa\,.$$
The regularity of $\kappa$ then implies that we can choose $U_\alpha \in \calU$ that is
disjoint from $V_\beta$ for all $\beta < \alpha$. But no point of $U_\alpha$ is a
complete accumulation point of $\calU$ by our assumption, hence we may clearly
find a non-empty open $V_\alpha \subset U_\alpha$ for which
$|\{U \in \calU : U \cap V_\alpha \ne \emptyset\}| < \kappa$. Clearly, $V_\alpha$ is
disjoint from $V_\beta$ and hence $U_\alpha \ne U_\beta$ for all $\beta < \alpha$.
This shows that our inductive procedure can be completed.
\end{proof}

The proofs of the following two propositions are straightforward
and hence are left to the reader.

\begin{pro}\label{pro:wll}
Suppose that $X$ is a weakly linearly Lindel\"of space. Then

\pritem{(a)}  any locally countable family $\calU\C \tau^*(X)$ is
countable;

\pritem{(b)} collectionwise normality of $X$ implies $ext(X)\leq\o$;

\pritem{(c)} every  continuous  image  of  $X$
is weakly linearly Lindel\"of;

\pritem{(d)} every  regular closed subspace of $X$
is weakly linearly Lindel\"of;

\pritem{(e)} if $X$ is a dense subspace in a space $Y$, then $Y$  is  weakly
linearly Lindel\"of;

\pritem{(f)} if $K$ is compact and Hausdorff, then $X\times K$ is
weakly linearly Lindel\"of;

\pritem{(g)}  every  perfect  irreducible   preimage  of  a  weakly
linearly Lindel\"of space is weakly linearly Lindel\"of.
\end{pro}

\begin{pro}
If  $X$  is  a  space such that $X=\bigcup_{n\in\o}X_n$ and every
$X_n$ is weakly linearly Lindel\"of,  then so is $X$.
\end{pro}

\begin{exa} \rm
Under CH,  Michael  constructed  in  \cite{Mi}  an  example  of a
regular Lindel\"of space $X$ such that $X\times X$ is paracompact
but not Lindel\"of.  As a consequence, there  exists  a  discrete
uncountable  family  of  non-empty  open  subsets in $X\times X$.
Applying Proposition \ref{pro:wll} (b), we  conclude that $X\times X$
is not weakly linearly Lindel\"of. This shows that {\em compact}
cannot be replaces by {\em Lindel\"of} in Proposition \ref{pro:wll} (f).
\end{exa}

\begin{thm}\label{tm:wlL}
(a)   Every   linearly   Lindel\"of   space  is  weakly  linearly
Lindel\"of;

\noindent(b)
every almost discretely Lindel\"of space is weakly  linearly
Lindel\"of;

\noindent (c)
every weakly Lindel\"of space is weakly  linearly
Lindel\"of.
\end{thm}

\begin{proof}
(a) and (b) Assume that the space $X$ is linearly Lindel\"of
or almost discretely Lindel\"of. Clearly, then
every {\em discrete} subset of $X$ of uncountable regular cardinality has a
complete accumulation point.
Now take  any
disjoint  family  $\calU=\{U_\alpha:  \alpha<\kappa\}\C \tau^*(X)$
for  some  uncountable  regular  cardinal  $\kappa$ and for  every
$\alpha<\kappa$ pick a  point  $x_\alpha\in U_\alpha$. Then
the, clearly discrete, set  $\{x_\alpha:   \alpha<\kappa\}$  has   a
complete accumulation point $x$ by the above.
It  is  obvious then that $x$ is also a complete accumulation point of
the family $\calU$   and   therefore, by   Proposition
\ref{pro:disj}, the proof is completed.

(c) Suppose that $X$  is  a  weakly Lindel\"of space and $\calU\C
\tau^*(X)$ is a family  of  an  uncountable  regular  cardinality
$\kappa$  that  has no complete accumulation point. For every $x\in
X$ we can take  a  set   $V_x\in   \tau(x,X)$   such  that  the  family
$\calP_x=\{U\in \calU:  U\cap V_x\neq\e\}$ has  cardinality  less
than $\kappa$.  For the open cover $\{V_x: x\in X\}$ of the space
$X$  we can find a countable set $B\C X$ such that $\bigcup\{V_x:
x\in  B\}$  is  dense  in  $X$.   As  an  immediate  consequence,
$\bigcup\{\calP_x: x\in B\} =\calU$ which contradicts the
regularity of $\kappa$  because  $|\calP_x|<\kappa$ for all $x\in
B$.

Our proof actually shows that in a  weakly
Lindel\"of  space  $X$  every family  $\calU\C \tau^*(X)$ whose
cardinality  has  uncountable  cofinality admits  a
complete accumulation point.
\end{proof}

Now we turn to presenting our main result formulated in the title of
our paper. An earlier version of the result used the following lemma that, however,
was replaced by the use of part (b) of Proposition \ref{pro:wll}.
Still, we decided to keep it because we think it has some independent interest.

\begin{lem} \label{lem:extcn}
Assume that $X$ is  a  collectionwise  normal  space and $Y$ is a
dense  subspace   of   $X$.    Then
$ext(X)\leq ext(Y)$.
\end{lem}

\begin{proof}
Suppose that $ext(Y)=\kappa$ and  $D\C X$ is a closed  discrete
subspace  such that $|D|=\kappa^+$.  Then there exists a discrete
family $\calU=\{U_d:  d\in  D\}\C  \tau(X)$  such that $d\in U_d$
for any $d\in D$. Pick a point  $x(d)\in  U_d\cap  Y$  for  every
$d\in  D$.   Then  $D'=\{x(d):   d\in  D\}$  is a closed discrete
subset of  $Y$  with  $|D'|=\kappa^+$  which  is a contradiction.
\end{proof}

The following example shows that we cannot replace the extent with
the Lindel\"of number in   Lemma \ref{lem:extcn}.

\begin{exa}\rm
If   $X=\{x\in   \amsD^{\o_1}:    |x^{-1}(1)|\leq\o\}$   is   the
$\Sigma$-product in $\amsD^{\o_1}$, then $X$ is a  collectionwise
normal  non-Lindel\"of  space  (see   Problem  102  of  the  book
\cite{Tkb3}) which has the dense $\sigma$-compact subspace  $S=\{x\in
\amsD^{\o_1}: |x^{-1}(1)|<\o\}$.

\end{exa}

\medskip

\begin{thm}\label{tm:main}
Every  monotonically normal and weakly linearly Lindel\"of space
is Lindel\"of.  \end{thm}

\begin{proof}
Let $X$ be a  monotonically  normal and weakly  linearly  Lindel\"of
space.    Then $X$  is  collectionwise  normal, hence we  can  apply part (b) of
Proposition \ref{pro:wll}  to  conclude  that  $ext(X)\leq  \o$. But
every paracompact space of countable  extent  is  Lindel\"of, so it
suffices to prove that $X$ is paracompact.

If  $X$ is not paracompact, then we can apply the celebrated characterization theorem of Balogh
and Rudin (see  \cite[Theorem  I]{BR})  to  conclude  that  there
exists  a  closed set $F\C X$ homeomorphic to a stationary subset
of some uncountable regular cardinal $\kappa$.  The set $F$ being scattered,
we can choose a discrete subspace $D\C F$ such that $\CL  D=F$.
Note that we have $|D| = |F| = \kappa$.

Clearly, the  set  $D$  has no complete accumulation
point in $F$ and hence in $X$.
Thus, for any point $x\in X$ we may pick an open neighborhood $W_x\in \tau(x,X)$ such that
$|D \cap W_x|<\kappa$.  It follows from the hereditary  collectionwise
normality  of  $X$  that we can find a set $V_d\in \tau(d,X)$ for
every $d\in D$ such that $V_d\cap (F\mi D)=\e$ and
the family $\{V_d: d\in D\}$ is disjoint.

Fix a monotone normality operator $O$ for $X$ and consider the
family $\calU=\{O(d,V_d):  d\in  D\}\C  \tau^*(X)$. Then $\calU$
is disjoint and its cardinality is equal to $\kappa$. If $x\in  V_d$
for  some  $d\in  D$,  then  $V_d$ trivially witnesses  that  $x$ is not a
complete  accumulation   point   of   $\calU$.    If, on the other hand,  $x\in  X\mi
\bigcup\{V_d: d\in D\}$, then   $O(x,W_x)  \cap
O(d,V_d)\neq\e$ must imply that  $d\in  W_x$, consequently we have

\medskip
\centerline{$\{d\in D: O(x,W_x)
\cap  O(d,V_d)   \neq\e\}   \C   D\cap W_x$.}

\medskip\noindent
But then, as  $|D\cap W_x|<\kappa$, we conclude that $x$ is not a
complete  accumulation  point of $\calU$. As a consequence, the  family  $\calU$  has  no
complete  accumulation  point  in  $X$, contradicting
the weak linear Lindel\"ofness of the space $X$
and thus completing the proof. \end{proof}

From Theorems \ref{tm:wlL} and \ref{tm:main} we immediately get the following.

\begin{cor} \label{cor:densel}
Suppose that $X$ is a  monotonically normal space possessing any one
of the following properties: weak
Lindel\"ofness,   linear    Lindel\"ofness, or almost    discrete
Lindel\"ofness. Then $X$ is Lindel\"of. \end{cor}

Moreover, from part (e) of Proposition \ref{pro:wll} we may obtain the following
statement that is formally stronger than Theorem \ref{tm:main}.

\begin{cor}
If a monotonically normal  space  has a dense weakly linearly Lindel\"of
subspace, then it is Lindel\"of. \end{cor}

It is well-known that spaces of countable cellularity, i.e.
spaces $X$ with $c(X)=\o$, are examples of weakly Lindel\"of spaces
that are not necessarily Lindel\"of.  However, we have $c(X)=hl(X)$
for  every  monotonically normal space $X$ (see e.g., \cite[Theorem
A]{Ga}), consequently Theorem \ref{tm:main} says nothing new for
monotonically normal spaces of countable cellularity.    However,   Corollary
\ref{cor:densel} seems to give us new information even for linearly ordered spaces:

\begin{cor}
Every weakly Lindel\"of   GO space  is Lindel\"of.
\end{cor}

\begin{cor}
If  a  GO space  $X$  has  a dense linearly Lindel\"of
subspace, then $X$ is Lindel\"of.
\end{cor}

\begin{exa}\rm
Any countably compact but non-compact S-space is an
example of  an
almost   discretely   Lindel\"of   space  that  is  not  linearly
Lindel\"of. Such examples are the HFD space constructed from CH in \cite{Juh}
that is hereditarily collectionwise normal and Ostaszewski's space constructed in \cite{Os}
from Jensen's  Axiom  $\diamondsuit$ that is even perfectly normal (and hence  first
countable).  While both of these examples have countable tightness, we  recall that
every discretely Lindel\"of space  of countable tightness must be
Lindel\"of by \cite[Corollary 3.5]{AB}. These examples also show that monotone
normality cannot be weakened essentially in Theorem \ref{tm:main}.
\end{exa}

\smallskip

We do not know whether an  example of  an
almost   discretely   Lindel\"of   but  not  linearly
Lindel\"of space can be given in ZFC.

\medskip

To conclude our paper, we now present a couple of
results that involve co-diagonals and the
almost discrete Lindel\"ofness property.

\begin{thm}\label{thm:smalld}
Suppose  that  $X$  is  a  regular  space  and its co-diagonal
$\Delta^c_X$ has the following properties:

\smallskip\noindent
(a) The closure of any  countable discrete subset of $\Delta^c_X$
in $\Delta^c_X$ is Lindel\"of;

\smallskip\noindent
(b) $\Delta^c_X$  is
almost discretely   Lindel\"of.

\smallskip\noindent
Then  $X$  has  small  diagonal.
\end{thm}

\begin{proof}
Consider any set $A\C \Delta^c_X$ with $|A|=\o_1$. We can
assume that

\smallskip
{
\pritem{$(*)$}
$A$ has no  complete accumulation points in $\Delta^c_X$.

}

\smallskip
\noindent
Insdeed, if $z \in \Delta^c_X$ is
a complete accumulation point point of $A$ in $\Delta^c_X$, then
we may choose a set $U\in \tau(z,X\times
X)$ such that $\CL U \cap \Delta_X=\e$. (The closure is taken in $X \times X$.)
But then $B=\CL U\cap A \C A$
is uncountable and $\CL B\cap \Delta_X =\e$.

Take  any point $x_0\in A$ and a set $U_0\in \tau(x_0,X\times X)$
such that $\CL U_0 \cap \Delta_X=\e$ and $U_0\cap A$ is countable.
Proceeding by induction  assume  that  $\alpha<\o_1$  and we have
a  set  $D_\alpha=\{x_\beta:  \beta  <\alpha\}  \C \Delta^c_X$
and a family $\{U_\beta:   \beta<\alpha\}\C  \tau(X\times
X)$ with the following properties:

\smallskip
{
\pritem{(4)}   $x_\beta\in   U_\beta   \cap   A$,  the  set  $\CL
U_\beta\cap A$ is countable and  $\CL U_\beta \cap \Delta_X=\e$ for
every $\beta<\alpha$;

\pritem{(5)}  $x_\beta\notin   Q_\beta=   \bigcup\{   U_\gamma:
\gamma<\beta\} \cup \CL{\{x_\gamma: \gamma<  \beta\}}$  for  each
$\beta <\alpha$.

}

\smallskip
If the set  $Q_\alpha=  \bigcup\{ U_\gamma:  \gamma<\alpha\} \cup
\CL{\{x_\gamma:   \gamma<  \alpha\}}$  covers   $A$,   then   the
induction procedure stops.   If  not,  then  we  pick a point
$x_\alpha\in  A\mi  Q_\alpha$  and  choose   a   set   $U_\alpha\in
\tau(x_\alpha,   X\times   X)$   such  that  $\CL  U_\alpha  \cap
\Delta_X=\e$ and $\CL U_\alpha  \cap  A$  is countable. It is clear
that  conditions  (4)  and  (5)  are   satisfied   for   all
$\beta\leq\alpha$.

Let us observe  first  that  every  set  $D_\alpha$ constructed above is
discrete because for any ordinal $\beta<\alpha$ we have
$$U_\beta  \cap  \big( X\times X\,\mi \CL{\{x_\gamma:
\gamma<\beta\}}\big) \cap D_\alpha  = \{x_\beta\}\,.$$

Now, assume  first
that  our  inductive  procedure ended at some step $\alpha<\o_1$,
i.e. $A \subset Q_\alpha$. Then $A \cap \bigcup\{ U_\gamma:  \gamma<\alpha\}$
is countable, hence
this implies that $A \cap \CL  D_\alpha$ is uncountable.  But by condition (a)
$\CL D_\alpha \cap \Delta^c_X$ is Lindel\"of, hence the  set  $A$  must
have   a  complete  accumulation  point  in  $\CL  D_\alpha  \cap
\Delta^c_X$,  contradicting  $(*)$.

If on the other hand our  procedure  lasts  $\o_1$-many  steps,   then   we   have
constructed  an uncountable  discrete  set  $D=\{x_\alpha:  \alpha<\o_1\}\C A$.
Then, by condition (b),   there  exists  a  Lindel\"of  set  $L\C
\Delta^c_X$ such that $D\C L$ and
so the uncountable set $D\C A$ must have a  complete
accumulation   point   in  $L$  and  hence in $\Delta^c_X$. But this again contradicts
$(*)$. \end{proof}

\begin{cor}\label{cor:smalld}
If $X$ is  a  regular  space  and its co-diagonal  $\Delta^c_X$ is
discretely Lindel\"of, then $X$ has  small diagonal. \end{cor}

\begin{cor}
Suppose that  $X$ is an $\o$-monolithic regular space such that
$\Delta^c_X$ is almost discretely Lindel\"of. Then
$X$ has  small diagonal. \end{cor}

The following statement gives a positive answer  to  Problem  4.6
from  the  paper  \cite{BT} and Question 5.8 from \cite{AJTW}.

\begin{cor}
If $X$ is a \LSS-space such that
$\Delta^c_X$  is  discretely  Lindel\"of,   then  $X$ has  a  countable
network.
\end{cor}

\begin{proof}
Given a compact $K\C X$  observe  that $\Delta^c_K$
is discretely Lindel\"of being a closed subspace of
$\Delta^c_X$.  By  \cite[Proposition 3.3]{BT}, the space $K$ is
metrizable.  Since all compact subsets  of $X$ are metrizable and
$X$ has small diagonal  by  Corollary  \ref{cor:smalld},  we  can
apply  \cite[Theorem  2.1]{Gr}  to  conclude  that  $X$  has  a
countable network. \end{proof}

\begin{cor}
Suppose that $X$ is an $\o$-monolithic compact Hausdorff space of
countable  tightness.   If  $\Delta^c_X$ is almost
discretely Lindel\"of, then $X$  is metrizable. In particular, if
$X$ is a Corson compact space and $\Delta^c_X$  is
almost discretely Lindel\"of, then $X$ is metrizable. \end{cor}

\begin{proof}
It follows from Theorem \ref{thm:smalld} that the space $X$ has
small   diagonal.   But  any  compact  $\o$-monolithic  space  of
countable  tightness  with  a  small  diagonal  is  metrizable by
\cite[Problem 296]{Tkb2}. \end{proof}

\begin{exa} \rm
In \cite{dVK}, Kunen and de  la  Vega  constructed  under  CH
a non-metrizable compact Hausdorff space $X$ such that $X^n$ is
hereditarily separable  for  all  $n\in  \amsN$.   In particular, then
$\Delta^c_X$  is  hereditarily   separable   and   hence   almost
discretely   Lindel\"of.  Thus,  almost  discrete  Lindel\"ofness  of
$\Delta^c_X$ does not imply the metrizability of a compact space $X$,
at least consistently.

It is easy to see that for the Kunen--de  la  Vega example $X$
we have $w(X)=\o_1$. Since
every compact space having  small  diagonal and weight
$\o_1$ is metrizable, the space $X$  cannot have small diagonal
(see  Problem 295 of the book \cite{Tkb2}).  Therefore, under CH, the
almost discrete  Lindel\"ofness  of  $\Delta^c_X$  does not imply that
$X$ has small diagonal for a compact Hausdorff space $X$. \end{exa}

We do not know whether there exists in ZFC a
non-metrizable compact space $X$ whose co-diagonal $\Delta^c_X$ is almost
discretely Lindel\"of.

\section{Open questions}

Since  it  is  still  not known whether every discretely
Lindel\"of space is Lindel\"of,  it  is  important to find out in
which stituations discrete Lindel\"ofness implies  Lindel\"ofness
or some weaker property. The list of respective open questions is
given below.

\begin{que}
Is every almost discretely Lindel\"of space weakly Lindel\"of?
\end{que}

\begin{que}
Is every discretely Lindel\"of space weakly Lindel\"of?
\end{que}

\begin{que}
Suppose  that $X$ is a space such that $\Delta^c_X$
is discretely $\sigma$-compact.  Is  it  true  that $X$ is
cosmic? \end{que}

\begin{que}
Suppose  that $X$ is a pseudocompact space such that $\Delta^c_X$ is discretely  Lindel\"of.  Is  it true in ZFC that
$X$ is compact and metrizable? \end{que}

\begin{que}
Suppose  that  $X$  is   an  almost  discretely  Lindel\"of  first
countable space. Then $|X|\leq  2^\cont$  but  is  it  true  that
$|X|\leq\cont$?
\end{que}

\begin{que}
Is there a ZFC example of an almost  discretely  Lindel\"of  space
that is not Lindel\"of?
\end{que}

\begin{que}
Suppose  that  $\Delta^c_X$  is  a  discretely
Lindel\"of Tychonoff space. Is it true in ZFC  that then $X$ is Lindel\"of?
and $iw(X)\leq\o$? \end{que}

\begin{que}
Suppose   that   $\Delta^c_X$  is  a  discretely
Lindel\"of Tychonoff space. Is it true in  ZFC  that  $X$  has  a
$G_\delta$-diagonal? \end{que}

\begin{que}
Suppose that $\Delta^c_X$ is discretely Lindel\"of.
Is it true in ZFC that $X$ is Lindel\"of? \end{que}

\begin{que}
Suppose that $\Delta^c_X$ is discretely Lindel\"of.
Is it true in ZFC that $|X|\leq\cont$? \end{que}

\begin{que}
Suppose that $X$ is a monotonically normal space and
$Y\C X$ is dense in $X$. Is it true that $l(X) \leq l(Y)$?
\end{que}

\begin{que}
Suppose  that  $X$  is discretely Lindel\"of and $K$ is a compact
space. Must $X\times K$ be discretely Lindel\"of? \end{que}

\medskip
\noindent{\bf \large  Acknowledgement.}  The  first author thanks
the   Department   of   Mathematics  of  Autonomous  Metropolitan
University of Mexico  for  the  support  during  the work on this
paper.

\bigskip
{\small
\noindent \rlap{\scshape Juh\'asz, Istvan}

\smallskip
\noindent \rlap{Alfr\'ed R\'enyi Institute of Mathematics,}

\noindent \rlap{Hungarian Academy of Sciences,}

\noindent  \rlap{Re\'altanoda u. 13--15,}

\noindent  \rlap{1053 Budapest, Hungary,}

\noindent  \rlap{\bf e-mail:  juhasz@renyi.hu}

}

\bigskip

{\small
\noindent \rlap{\scshape Tkachuk, Vladimir V.}

\smallskip
\noindent \rlap{Departamento de Matem\'aticas,}

\noindent \rlap{Universidad Aut\'onoma Metropolitana,}

\noindent  \rlap{Av.  San  Rafael  Atlixco,   186,  }

\noindent  \rlap{Col. Vicentina, Iztapalapa }

\noindent  \rlap{C.P. 09340,  Mexico D.F., Mexico}

\noindent  \rlap{\bf e-mail:  vova@xanum.uam.mx}
}

\bigskip
\bigskip
\bigskip

{\small
\noindent \rlap{\scshape Wilson, Richard G.}

\smallskip
\noindent \rlap{Departamento de Matem\'aticas,}

\noindent \rlap{Universidad Aut\'onoma Metropolitana,}

\noindent  \rlap{Av.  San  Rafael  Atlixco,   186,  }

\noindent  \rlap{Col. Vicentina, Iztapalapa }

\noindent  \rlap{C.P. 09340,  Mexico D.F., Mexico}

\noindent  \rlap{\bf e-mail:  rgw@xanum.uam.mx}
}

\end{document}